\documentclass[a4paper,12pt]{amsart}
\usepackage{graphicx}
\usepackage{amsmath,amssymb,amsthm,amsfonts}
\usepackage{xcolor}
\usepackage{amssymb}
\usepackage[active]{srcltx}

\newtheorem{thm}{Theorem}[section]
\newtheorem{cor}[thm]{Corollary}
\newtheorem{lem}[thm]{Lemma}
\newtheorem{prop}[thm]{Proposition}
\newtheorem{rem}[thm]{Remark}
\newtheorem{defn}[thm]{Definition}
\newcommand{\R}{{\mathbb{R}}}

\begin{document}

\parindent 0pc
\parskip 6pt
\overfullrule=0pt

\title[Pohozaev identity]{On the Pohozaev  identity for  quasilinear  Finsler anisotropic equations}

\author{Luigi Montoro$^{*}$ and Berardino Sciunzi$^{*}$}
\address{$^{*}$Dipartimento di Matematica e Informatica, UNICAL, Ponte Pietro  Bucci 31B, 87036 Arcavacata di Rende, Cosenza, Italy}
\email{montoro@mat.unical.it, sciunzi@mat.unical.it}

\keywords{Pohozaev identity, Finsler anisotropic operator,  quasilinear equations}

\subjclass[2020]{35A23, 35A01}

\maketitle

\begin{abstract}
In this paper we derive the Pohozaev    identity  for quasilinear equations 
\begin{equation}\tag{$E$}\label{eq:p}
-\operatorname{div}(B'(H(\nabla u))\nabla H(\nabla u))=g(x, u) \quad \text {in}\,\, \Omega,
\end{equation}
involving the  anisotropic Finsler operator $-\operatorname{div}(B'(H(\nabla u))\nabla H(\nabla u))$.
In particular, by means   of  fine regularity results on  the vectorial field   $B'(H(\nabla u))\nabla H(\nabla u)$,  we  prove the identity for  weak solutions and in a direct way.
\end{abstract}
\section{Introduction and main results}
In this paper we derive  the  Pohozaev identity for the Finsler anisotropic operator, defined for suitable smooth functions as
\begin{equation}\label{eq:finsleroperator}
\operatorname{div}(B'(H(\nabla u))\nabla H(\nabla u)),\end{equation}
where the functions $ B, H$ are defined  below.

 As well known, in 1965 Pohozaev in his seminal paper \cite{P} considered the following Dirichlet problem
\begin{equation}\label{eq:ppoh}
\begin{cases}
-\Delta u=g(u) & \text{in}\,\, \Omega
\\
u=0 &\text{on}\,\, \Omega,
\end{cases}
\end{equation}
with $\Omega$ a bounded smooth domain of $\mathbb R^N$ and $g$ a continuous function on $\mathbb R$.
He proved  that, for classical   solutions $u\in C^2(\Omega)\cap C^1(\overline \Omega)$ to \eqref{eq:ppoh}, the following identity holds
\begin{equation}\label{eq:ppoh1}
(2-N)\int_{\Omega}ug(u)\, dx+2N\int_{\Omega}G(u)\, dx=\int_{\partial \Omega}|\nabla u|^2(x,\eta)\,dx,
\end{equation}
where $G'=g$ with $G(0)=0$ and $\eta$ is the outward unit normal vector to the boundary $\partial \Omega$.  Important extensions and further developments can be found in the literature. We recall here the first result in the quasilinear setting \cite{pucser1}   where the authors obtained a corresponding identity for smooth $C^2$ extremals of general variational problems, including results for systems and higher order operators. Since in general, e.g. for the $p$-Laplacian, solutions are not of class $C^2$, we may  consider very important the improvement of the results in \cite{pucser1}  obtained in \cite{DMS},  where the authors proved the Pohozaev identity for  $C^1$ solutions  of general quasilinear  Dirichlet problems, by mean of a  suitable and technical approximation argument. See also the contribution in \cite{GV}. Although to our purposes  we could also try to follow the technique of \cite{DMS}, we have to consider the fact that the adaptation to the anisotropic case of such a technique is completely non trivial. On the contrary, by means of fine regularity results, we shall follow a more simple a direct proof.

The main consequence of \eqref{eq:ppoh1} is the nonexistence of nontrivial solution to \eqref{eq:ppoh} when $g$ satisfies suitable assumptions and $\Omega$ is star-shaped with  respect to the origin. This is quite a delicate issue that we will discuss for the anisotropic case in Section 2. The typical nonexistence result for the critical case is contained in  Corollary \ref{cor:pohozaev}.

From now on  we   consider the following quasilinear anisotropic elliptic equation
\begin{equation}\tag{$E$}\label{eq:p}
-\operatorname{div}(B'(H(\nabla u))\nabla H(\nabla u))=g(x, u) \quad \text {in}\,\, \Omega,
\end{equation}
where $\Omega\subseteq\mathbb R^N$ is a  smooth domain  and   $g$ is a nonlinearity satisfying the following assumption
 \begin{itemize}
 \item[$(h_g)$]
 $g:\overline\Omega\times \mathbb R\rightarrow R$ is a $C^1$ function on the domain $\overline \Omega\times \mathbb R$.
\end{itemize}	
From the mathematical point of view, the anisotropy is responsible for a more richer geometric structure than the classical Euclidean case. On the other hand different
 applications come from  several real phenomena where anisotropic media naturally arise.

In all the paper we suppose the following  classical structural assumptions  on the anisotropic operator. We assume that the functions $B$ and $H$ satisfy
\begin{itemize}
	\item[$(h_B)$]
	\begin{enumerate}
		\item[(i)] $B \in C^{3,\beta}_{loc}((0,+\infty)) \cap C^1([0,+\infty))$, with $\beta \in (0,1)$;
		
		\item[(ii)] $B(0)=B'(0)=0$, $B(t), B'(t), B''(t)>0 \; \forall t \in (0,+\infty)$;
		
		\item[(iii)] There exists $p>1$, $\kappa \in [0,1]$, $\gamma > 0$, $\Gamma > 0$ such that
		\begin{equation}\label{hp:boundAboveBelow}
			\begin{split}
				\gamma(\kappa+t)^{p-2}t \leq & B'(t) \leq \Gamma (\kappa+t)^{p-2}t \\
				 \gamma(\kappa+t)^{p-2} \leq & B''(t) \leq \Gamma (\kappa+t)^{p-2}
			\end{split}
		\end{equation}
	     for any $t>0$;
	\end{enumerate}

\

and

\
	\item[$(h_H)$]
\begin{enumerate}
	\item[(i)] $H \in C^{3,\beta}_{loc}(\R^N \setminus \{0\})$ and such that $H(\xi)>0 \quad \forall \xi \in \R^N \setminus \{0\}$;
	
	\item[(ii)] $H(s \xi) = |s| H(\xi) \quad \forall \xi \in \R^N \setminus \{0\}, \, \forall s \in \R$;
	
	\item[(iii)] $H$ is {uniformly elliptic}, that means the set
	$$\mathcal{B}_1^H:=\{\xi \in \R^N  :  H(\xi) < 1\}$$ is uniformly convex, i.e.
	\begin{equation}\label{eq:boyuH}
        \exists \Lambda > 0: \quad \langle D^2H(\xi)v, v \rangle \geq \Lambda |v|^2 \quad \forall \xi \in \partial \mathcal{B}_1^H, \; \forall v \in \nabla H(\xi)^\bot.
	\end{equation}
\end{enumerate}
\end{itemize}
Every anisotropy $H$ having the unit ball  $\mathcal{B}_1^H$ uniformly convex is called uniformly elliptic. Indeed, the second fundamental form of $\partial \mathcal{B}_1^H$ at a point $\xi \in \partial \mathcal{B}_1^H$ is given by
\[\frac{\langle D^2H(\xi)l, v \rangle}{|\nabla H(\xi)|}\quad  \forall \,l,v \in \nabla H(\xi)^\bot.\]
Since $\partial \mathcal{B}_1^H$ is compact, the uniform ellipticity of $H$ is equivalent to ask \eqref{eq:boyuH}. For details we refer to \cite{CFV1,CFV}.

Under our assumptions the natural function space for solutions to equation \eqref{eq:p} is $W^{1,p}(\Omega)\cap L^{\infty}(\Omega)$. Indeed the anisotropic operator $-\operatorname{div}(B'(H(\nabla u))\nabla H(\nabla u))$, when  $\kappa=0$ in $(h_B)$,   becomes degenerate ($p>2$) or singular ($1<p<2$) in the critical set
\begin{equation}\label{eq:criticalsetZ}\mathcal Z=\{x\in \Omega \,:\, \nabla u(x)=0\}\end{equation} and solutions are not classical in general. We remark at this stage  that, as a particular special case, the anisotropic operator  contains  the well known $p$-laplace operator:  as well known \cite{DB,Tolk}, already in this case the optimal regularity of  solutions to quasilinear equations  with the  $-\Delta_p$ operator is  $C^{1,\alpha}$.

In \cite{CFV} the authors show  that it is possible the apply the  results in  \cite{DB,Tolk} to show   that weak solutions to \eqref{eq:p} are $C^{1,\alpha}(\Omega)\cap C^2(\Omega\setminus Z)$ for some $0<\alpha<1$.

We notice that, in general, the optimal regularity for  solutions to this type of problems is known to be $C^{1,\alpha}$, see  the recent paper \cite{ACF} about the interior regularity  for weak solutions to anisotropic quasilinear equations. Regarding the regularity of the second derivatives  we mention that first results have been obtained in \cite{CSR}.  In particular it has been shown in \cite{CSR} that the stress field $B'(H(\nabla u))\nabla H(\nabla u)\in W^{1,2}_{\rm loc}(\Omega)$ among some very stronger regularity estimates that lead to the optimal summability of the second derivatives. In the case $B(t)=\frac{t^p}{p}$ the fact that
$H^{p-1}(\nabla u)\nabla H(\nabla u)\in W^{1,2}_{\rm loc}(\Omega)$ has been also proved recently in \cite{ACF} under suitable more general assumptions on the source term. 

Taking into account these facts, the goal of this paper is twofold: on one hand is to get a  Pohozaev identity for quasilinear problems involving the general anisotropic operator \eqref{eq:finsleroperator}; on  the other hand is to get a general Pohozaev identity for $C^1$ weak solutions to \eqref{eq:p} and  then avoid some possible restriction coming eventually from asking more regularity on the solution to  \eqref{eq:p}.  To the best of our knowledge  this identity is new  for quasilinear equations driven by a general Finsler anisotropic operator in  the framework of weak solutions.

We mention  \cite{WZ} where the authors proved a Pohozaev identity for a quasilinear  problem involving the anisotropic $p$-Laplace operator ($B(t)=t^p/p$) and assuming $C^2$ regularity of the solutions.

We are ready now   to say   what we mean for   $C^1$ weak solutions to \eqref{eq:p}. More specifically we give the following
\begin{defn}\label{def:sol}
Let us assume that $(h_g), (h_B)$ and $(h_H)$ hold. Then $u \in C^1(\Omega)$ is a weak solution of \eqref{eq:p}  if
\begin{equation}\label{eq:weakDoublyCritical}
		\int_{\Omega} B'(H(\nabla u))\langle \nabla (H(\nabla u)) , \nabla \varphi \rangle  \, dx= \int_{\Omega} g(x,u)\varphi \, dx \quad \forall \varphi \in C^\infty_c(\Omega).
\end{equation}
\end{defn}
In the following we use $\eta(x)=(\eta_1,\ldots,\eta_n)(x)$ to denote the unit outward normal vector to $\partial \Omega$ at a point $x\in \partial \Omega$. We will also use the notation \[u_\eta(x):=\frac{\partial u(x)}{\partial \eta}.\]
Trough  the paper  we  define the functions $G$ and $G_{x_i}$ on the domain $\overline \Omega \times \mathbb R$ as
\begin{equation}\label{eq:G}
G(x,t)=\int_0^t g(x,s)\, ds \qquad G_{x_i}(x,t)=\int_0^t g_{x_i}(x,s)\, ds,
\end{equation}
where $g_{x_i}(x,s)=\partial g(x,s)/ \partial x_i$.

Here below we state our main results
\begin{thm}[Pohozaev identity]\label{thm:pohozaev}
Let us assume that $(h_g), (h_B)$ and $(h_H)$ hold. Let $\Omega$ a bounded smooth domain and let  $u\in C^1(\overline \Omega)$ be a weak solution to \eqref{eq:p}. Then we have
\begin{eqnarray}\label{eq:pohozaev}
&& \int_{\Omega}NG(x,u) + (x\cdot \nabla_x G(x,u))-NB(H(\nabla u)) +B'(H(\nabla u))H(\nabla u) \, dx
\\\nonumber
&&=\int_{\partial \Omega}G(x,u)(x\cdot \eta) -B(H(\nabla u))(x\cdot\eta) +B'(H(\nabla u))(x\cdot \nabla u) (\nabla H(\nabla u)\cdot \eta)\, ds,
\end{eqnarray}
\end{thm}
We observe that if we take
$ B(t)={t^p}/{p}$ and $H(\xi)=|\xi|,$
 $p>1$, i.e. the case of the  $p$-Laplacian operator and $g(x,u)=g(u)$,  we recover the   well known classical  Pohozaev identity (see   \cite{DFSV,DMS,GV, MM}) given in the  following equation
\begin{eqnarray*}
 &&N\int_{\Omega}G(u)\, dx-\frac{N-p}{p}\int_{\Omega}|\nabla u|^p \, dx \\\nonumber
&&=\int_{\partial \Omega}\left(G(u)(x\cdot \eta)-\frac1p |\nabla u|^p(x\cdot\eta) + |\nabla u|^{p-2}(x\cdot \nabla u) u_\eta\right ) \, ds.
\end{eqnarray*}
Having proved Theorem \ref{eq:pohozaev} we can exploit it to deduce a result in the whole $\mathbb R^N$. We have the following
\begin{thm}\label{thm:R^N}
Let us assume that $(h_g), (h_B)$ and $(h_H)$ hold and let $u\in C^1(\mathbb{R}^N)$ a weak solution to
\begin{equation}\nonumber
-\operatorname{div}(B'(H(\nabla u))\nabla H(\nabla u))=g(x,u) \quad \text {in}\,\,\mathbb{R}^N,
\end{equation}
with $\kappa=0$ in $(h_B)$. Let us assume that
\begin{equation}\label{eq:radii}
B(H(\nabla u)), G(x,u),x\cdot\nabla_x G(x,u)\in L^1{(\mathbb R^N)}.\end{equation}
Then
\begin{equation}\nonumber
\int_{\mathbb{R}^N}NG(x,u)+ (x\cdot \nabla_x G(x,u))\, dx=\int_{\mathbb{R}^N}NB(H(\nabla u)) -B'(H(\nabla u))H(\nabla u) \, dx.
\end{equation}
\end{thm}
The rest of the paper is devoted to the proof of our main results and other consequences.

\section{The Pohozaev identity}
\noindent {\bf Notation.} Generic fixed and numerical constants will
be denoted by $C$ (with subscript in some case) and they will be
allowed to vary within a single line or formula. \\

The first step to show Theorem \ref{thm:pohozaev} is  recover a stronger formulation for weak  solutions to
\eqref{eq:p}. We do it in the next proposition. We also collect there some regularity results on the weak solution to \eqref{eq:p} that are interesting in itself. We have the  following
\begin{prop}\label{pro:a.e.}
Let us assume that $(h_g), (h_B)$ and $(h_H)$ hold. Let $u\in C^1(\Omega)$ be a weak solution to \eqref{eq:p}. Then\begin{itemize}
\item[$(i)$] The vector field $B'(H(\nabla u))\nabla H(\nabla u)$ satisfies
\begin{equation}\label{eq:W12}
B'(H(\nabla u))\nabla H(\nabla u)\in W^{1,2}_{\rm loc}(\Omega)
\end{equation}
and
\begin{equation}\label{eq:strongsol}
-\operatorname{div}(B'(H(\nabla u))\nabla H(\nabla u))=g(x,u)\,\,  \text{a.e. in }\Omega,\end{equation}
that is $u$  fulfills  \eqref{eq:p} in the classic sense for almost every $x\in \Omega$.
\item[$(ii)$] The weak  solution $u$ has second distributional derivatives and it holds
\[|\nabla u|^{p-2}\nabla u\in W^{1,2}_{\rm loc}(\Omega)\quad \text{and} \quad |\nabla u|^{p-1}\in W^{1,2}_{\rm loc}(\Omega).\]
Moreover
\[u\in W^{2,2}_{\rm loc}(\Omega),\quad \text{if}\,\, 1<p<3\]
and
\[u\in W^{2,m}_{\rm loc}(\Omega),\,\, 1\leq m<{(p-1)}/{(p-2)},\quad\text{if}\,\, p \geq3\quad\text{and}\quad  g(x,s)>0. \]
\end{itemize}
\end{prop}
\begin{proof} We prove part $(i)$ of the statement.
By \cite{CFV}, a $C^1$ weak solution $u$ to \eqref{eq:p}, is actually  regular outside the critical set $\mathcal Z$, see \eqref{eq:criticalsetZ}. That is $u\in C^2(\Omega \setminus \mathcal Z)$.  In what follows,  we denote by $\nabla u_{x_i}$ and $u_{x_ix_j}$ the second derivatives of $u$ outside the set $Z$. We mean them extended to zero on $Z$.
At the end of the proof we will see that these derivatives coincide with the distributional ones in $\Omega$.

For all $n>0$, let us define  $J_n:\mathbb{R}^+_0\rightarrow\mathbb{R}$ by setting
\begin{equation}\label{eq:J}
J_n(t)=\begin{cases}
t & \text{if  $\displaystyle t\geq \frac{2}{n}$} \\
\displaystyle2t-\frac{2}{n}& \text{if $\displaystyle\frac{1}{n}\leq t\leq\frac{2}{n}$}
\\ 0 & \text{if $\displaystyle0\leq t\leq \frac{1}{n}$}.
\end{cases}
\end{equation}
and let $h_n:\mathbb{R}^+_0\rightarrow\mathbb{R}$ defined as
\begin{equation}\label{eq:hn}
h_n(t)=\frac{J_n(t)}{t} \quad \text{and} \quad h_n(0)=0.
\end{equation}
Let us set
\[\varphi_n(x)=B'(H(\nabla u))H_{\xi_j}(\nabla u)h_n(|\nabla u|).\]
Using the definition of  $h_n$ (the function $h_n$ is actually zero in a neighborhood of $\mathcal Z$) we have that since $u\in C^2(\Omega \setminus \mathcal Z)$, then $h_n\in W^{1,2}_{\rm loc}(\Omega)$. Moreover
\begin{eqnarray}\label{eq:phin}
\frac{\partial \varphi_n}{\partial{x_i}}&=&
B''(H(\nabla u))\langle\nabla_{\xi} H(\nabla u),\nabla u_{x_i} \rangle H_{\xi_j}(\nabla u)h_n(|\nabla u|) \\\nonumber
&+&B'(H(\nabla u))\langle \nabla_\xi H_{\xi_j}(\nabla u), \nabla u_{x_i}\rangle h_n(|\nabla u|) \\\nonumber
&+&(B'(H(\nabla u))H_{\xi_j}(\nabla u))h'_n(|\nabla u|)\frac{\langle\nabla u,\nabla u_{x_i}\rangle}{|\nabla u|}\\\nonumber
&=&A_{n,1}+A_{n,2}+A_{n,3}.
\end{eqnarray}
for $i=\{1,\ldots,N\}$ that is well defined since $u\in C^2(\Omega \setminus \mathcal Z)$. Now we want to show that, for $E\subset\subset \Omega$, it holds that $\|\varphi_n\|_{W^{1,2}(E)}\leq C$, for some positive constant $C=C(E)$ not depending on $n$. To do this in the following we will use a local weighted estimate for the hessian of the solution $u$.
Under our hypotheses, for all $E\subset\subset \Omega$ and $0\leq \kappa\leq 1$,  a solution $u$ to \eqref{eq:p}   satisfies
the estimate
\begin{equation}\label{eq:localH}
\int_{E\setminus Z}(\kappa+|\nabla u|)^{p-2-\beta}\|D^2u\|^2\,dx,
\end{equation}
with  $\beta\in [0, 1)$ and where $D^2 u$ denotes the hessian of the solution $u$. The estimates we use here, are contained in the regularity results proved in \cite{CSR} and in a  forthcoming paper \cite{MS} in a more general setting.

First of all recalling that $H$ is 1-homogeneous, we have that $\nabla_{\xi}H$ is 0-homogeneous and therefore we obtain
\[\nabla_\xi H(\xi)=\nabla_\xi H\left(|\xi|\frac{\xi}{|\xi|}\right)=\nabla_\xi H\left(\frac{\xi}{|\xi|}\right), \]
for all $\xi\in \mathbb R^N\setminus \{0\}$.
For a similar argument, we get the next  estimate for the hessian of $H$, i.e.
\[D^2  H(\xi)=D^2  H\left(|\xi|\frac{\xi}{|\xi|}\right)=\frac{1}{|\xi|}\ D^2 H\left(\frac{\xi}{|\xi|}\right), \]
for all $\xi\in \mathbb R^N\setminus \{0\}$.

Then there exists a positive constant $M$ such that $|\nabla_\xi H(\xi)|\leq M$, $|D^2 H(\xi)|\leq M/ |\xi|$ for all $\xi\in \mathbb R^N\setminus \{0\}$, since $\nabla_\xi H,D^2 H$ are  continuous. We will use such estimates in the following computations.

Hence,
by $(h_B)$-$(iii)$ and the fact that $H(\xi)$ is a norm equivalent to the euclidian one, i.e.  there exist $c_1,c_2>0$ such that
\begin{equation}\label{eq:normequivbag}
c_1|\xi|\leq H(\xi)\leq c_2|\xi|,
\end{equation}
 we obtain that
\begin{eqnarray}\label{est1}
A_{n,1}^2&=&\big(B''(H(\nabla u))\langle\nabla_{\xi} H(\nabla u),\nabla u_i \rangle H_{\xi_j}(\nabla u)h_n(|\nabla u|)\big)^2\\\nonumber
&\leq&C (\kappa+|\nabla u|)^{2(p-2)}\|D^2u\|^2\chi_{E\setminus \mathcal Z},
\end{eqnarray}
where we used that $h_n(t)\leq1$ for all $t\geq 0$, where $\chi_{A}$ is the characteristic function of a measurable set $A$.

For $1<p<2$ we also have
\begin{equation}\label{eq:casep<2}
\int_{E}(\kappa+|\nabla u|)^{2(p-2)}\|D^2u\|^2\chi_{E\setminus \mathcal Z}\, dx\leq \|(\kappa+|\nabla u|)\|_{L^{\infty}(\Omega)}^{p-2+\beta} \int_{E\setminus \mathcal Z}(\kappa+|\nabla u|)^{p-2-\beta}\|D^2u\|^2\, dx.
\end{equation}
For $p\geq 2$ we have
\begin{equation}\label{eq:casep>2}
\int_{E}(\kappa+|\nabla u|)^{2(p-2)}\|D^2u\|^2\chi_{E\setminus \mathcal Z}\, dx\leq C(\kappa,p,\|\nabla u\|_{L^{\infty}(\Omega)})\|_{L^{\infty}(\Omega)}\int_{E\setminus \mathcal Z}(\kappa+|\nabla u|)^{p-2}\|D^2u\|^2\, dx.\end{equation}
Therefore using \eqref{eq:localH}, from \eqref{eq:casep<2} and \eqref{eq:casep>2} we deduce that $\|A_{n,1}\|_{L^2(E)}\leq C$, with $C$ not depending on $n$.

The second term $A_{n,2}$ can be estimated as follows
\begin{eqnarray}\label{est2}
&&A^2_{n,2}=\big (B'(H(\nabla u))\langle \nabla_\xi H_{\xi_j}(\nabla u), \nabla u_{x_i}\rangle h_n(|\nabla u|)\big)^2\\\nonumber
&&\leq\big (B'(H(\nabla u)) |\nabla_\xi H_{\xi_j}(\nabla u)| |\nabla u_{x_i}| h_n(|\nabla u|)\big)^2\\\nonumber
&&\leq C(\kappa +|\nabla u|)^{2(p-2)}|\nabla u|^2|\nabla u|^{-2}\|D^2u\|^2\chi_{E\setminus \mathcal Z}
\\\nonumber&&\leq (\kappa +|\nabla u|)^{2(p-2)}\|D^2u\|^2\chi_{E\setminus \mathcal Z}.
\end{eqnarray}
As for the case of $A_{n,1}$, we get that $\|A_{n,2}\|_{L^2(E)}\leq C$.

For the last term $A_{n,3}$ we have have the following
\begin{eqnarray}\label{est3-1}
&&A^2_{n,3}=\left((B'(H(\nabla u))H_{\xi_j}(\nabla u))h'_n(|\nabla u|)\frac{\langle\nabla u,\nabla u_{x_i}\rangle}{|\nabla u|}\right)^2\\\nonumber
&&\leq C(\kappa +|\nabla u|)^{2(p-2)}|\nabla u|^2(h'_n(|\nabla u|))^2\|D^2u\|^2\chi_{E\setminus \mathcal Z}.
\end{eqnarray}
Exploiting \eqref{eq:hn}, by a straightforward  computation
$$
h'_n(t)=
\begin{cases}
 0 & \text{if $\displaystyle t > \frac{2}{n} $}
\\
\displaystyle\frac{2}{nt^2}& \text{if $\displaystyle  \frac{1}{n} < t< \frac{2}{n} $}
\\ 0 & \text{if $\displaystyle 0\leq t< \frac{1}{n} $},
\end{cases}
$$
and then we have   $|\nabla u|h_n'(|\nabla u|)\leq 2$. Therefore from \eqref{est3-1}
\begin{eqnarray}\label{est3}
&&A^2_{n,3}\leq C(\kappa +|\nabla u|)^{2(p-2)}(|\nabla u|h'_n(|\nabla u|))^2\|D^2u\|^2\chi_{E\setminus \mathcal Z}\\\nonumber
&&\leq C(\kappa +|\nabla u|)^{2(p-2)}\|D^2u\|^2\chi_{E\setminus \mathcal Z}
\end{eqnarray}
and therefore, as above, we deduce that
$\|A_{n,3}\|_{L^2(E)}\leq C$.
Hence taking in to account the previous estimates from \eqref{eq:phin} we get
\begin{equation}\label{eq:phiw12}
\|\varphi_n\|_{W^{1,2}(E)}\leq C,
\end{equation}
with $C$ not depending on $n$. Therefore $\varphi_n \rightharpoonup \varphi_n\in W^{1,2}(E)$. Moreover because of the compact embedding in $L^2(E)$, up to a subsequence, $\varphi_n\rightarrow\varphi$ a.e. in $E$. On the other hand
\[\varphi_n \rightarrow B'(H(\nabla u))H_{\xi_j}(\nabla u)\]
and hence \[ B'(H(\nabla u))H_{\xi_j}(\nabla u)\equiv \varphi\in W^{1,2}(E).\]
Since $\xi_j$, $j=1,\cdots, N$ is arbitrary, we  have that $B'(H(\nabla u))\nabla_{\xi}H(\nabla u)\in W^{1,2}(E,\mathbb R^N)$, namely \eqref{eq:W12}.

Once we have  \eqref{eq:W12} integrating by parts  the left hand side of \eqref{eq:weakDoublyCritical} we get \eqref{eq:strongsol}.

\

The part $(ii)$ of the statement  follows   exploiting   arguments contained in \cite{DS1}. So we skip it.
\end{proof}
\begin{rem} Note that, if the nonlinearity  $g$ allows to exploit the Hopf Boundary Lemma (see \cite[Theorem 4.5]{CSR}), then the regularity results of Proposition~\ref{pro:a.e.}  can be extended to the clousure of $\Omega$. Moreover if $g(x,s)>0$  the critical set $\mathcal Z$ has zero Lebesgue measure, that is $|\mathcal Z|=0$. In this case the second distributional $u_{x_ix_j}$ coincide a.e. with the classical ones.
\end{rem}
In the proof of Theorem \ref{thm:pohozaev} we  shall use a very fine version of the divergence theorem  \cite[Lemma A.1.]{CT}. For the reader's convenience we state it.
\begin{lem}\label{lem:divthm}
Let $\Omega$ be a bounded domain in $\mathbb R^N$ with a $C^2$-boundary $\partial \Omega$. Assume that ${\bf a}\,:\, \overline \Omega \rightarrow \mathbb R^N $ satisfies ${\bf a}\in [C^0(\overline \Omega)]^N$ and $\operatorname{div}{\bf a}=f\in L^1(\Omega) $ in the sense of distributions in $\Omega$. Then we have
\[\int_{\partial \Omega}{\bf a}\cdot \eta\, ds=\int_{\Omega}f(x)\, dx.\]
\end{lem}

\begin{proof}[{Proof of Theorem~\ref{thm:pohozaev}}]
By Proposition \ref{pro:a.e.} we have that
\[-\operatorname{div}(B'(H(\nabla u))\nabla H(\nabla u))=g(x, u)\quad  \text{a.e. in} \,\,\Omega.\] Thus using the multiplier $(x\cdot \nabla u)$ in  both side of this equation and integrating we obtain
\begin{equation}\label{eq:campeur}
\int_{\Omega}\operatorname{div}(B'(H(\nabla u))\nabla H(\nabla u))(x\cdot \nabla u)\, dx=\int_{\Omega}-g(x,u)(x\cdot \nabla u)\, dx.
\end{equation}
Recalling \eqref{eq:G} we note now that
\begin{eqnarray}\label{eq:G1}
 &&-\int_{\Omega}g(x,u)(x\cdot \nabla u)\, dx=\sum_{i=1}^N\int_{\Omega}-x_i \partial_{x_i}G(x,u)\, dx+\sum_{i=1}^N\int_{\Omega} x_i, G_{x_i}(x,u)\, dx
 \\\nonumber&&=N\int_{\Omega}G(x,u)\, dx
-\int_{\partial \Omega}G(x,u)(x\cdot \eta)\, ds +\int_{\Omega} (x\cdot \nabla_x G(x,u))\, dx,
\end{eqnarray}
where in the last line we have used the divergence theorem, Lemma \ref{lem:divthm}.

We point out that integration in  the left hand side of \eqref{eq:campeur}, in particular the  divergence theorem, can  be applied thanks to the regularity result of the Proposition \ref{pro:a.e.}, in particular~\eqref{eq:W12}. Hence, taking also into account the   Euler's theorem for homogeneous functions, we obtain
\begin{eqnarray}\label{eq:cine1}
&&\quad \int_{\Omega}\operatorname{div}(B'(H(\nabla u))\nabla H(\nabla u))(x\cdot \nabla u)\, dx\\\nonumber
&&=-\int_{\Omega}B'(H(\nabla u))\langle \nabla H(\nabla u),\nabla (x\cdot \nabla u)\rangle \, dx+\int_{\partial \Omega}B'(H(\nabla u))(x\cdot \nabla u) \nabla H(\nabla u)\cdot \eta\, ds\\\nonumber
&&=-\int_{\Omega}B'(H(\nabla u))H(\nabla u) \, dx-\sum_{i=1}^N\int_{\Omega}B'(H(\nabla u))\langle \nabla H(\nabla u), \nabla u_{x_i}\rangle x_i\, dx\\\nonumber
&&+ \int_{\partial \Omega}B'(H(\nabla u))(x\cdot \nabla u) \nabla H(\nabla u)\cdot \eta\, ds\\\nonumber
&&=-\int_{\Omega}B'(H(\nabla u))H(\nabla u) \, dx-\sum_{i=1}^N\int_{\Omega}\partial_{x_i}B(H(\nabla u)) x_i\, dx\\\nonumber
&&+ \int_{\partial \Omega}B'(H(\nabla u))(x\cdot \nabla u) \nabla H(\nabla u)\cdot \eta\, ds\\\nonumber
&&=-\int_{\Omega}B'(H(\nabla u))H(\nabla u) \, dx+N\int_{\Omega}B(H(\nabla u)) \, dx\\\nonumber
&&-\int_{\partial \Omega} B(H(\nabla u))(x\cdot\eta)\, ds + \int_{\partial \Omega}B'(H(\nabla u))(x\cdot \nabla u) \nabla H(\nabla u)\cdot \eta\, ds,
\end{eqnarray}
where in the last line we applied the divergence theorem one more time.
From \eqref{eq:campeur}, collecting together \eqref{eq:G1} and \eqref{eq:cine1} we get
\begin{eqnarray*}
 &&N\int_{\Omega}G(u)\, dx +\int_{\Omega} (x\cdot \nabla_x G(x,u))\, dx -N\int_{\Omega}B(H(\nabla u)) \, dx +\int_{\Omega}B'(H(\nabla u))H(\nabla u) \, dx
\\\nonumber
&&=\int_{\partial \Omega}G(u)(x\cdot \eta)\, ds-\int_{\partial \Omega} B(H(\nabla u))(x\cdot\eta)\, ds + \int_{\partial \Omega}B'(H(\nabla u))(x\cdot \nabla u) \nabla H(\nabla u)\cdot \eta\, ds,
\end{eqnarray*}
namely our thesis.
\end{proof}
As well known, the domain it is called star-shaped with respect to the origin if for every $x\in \Omega$ the line segment $\overline{0x}$ is contained in $\Omega$. Moreover it is easy to see that $x\cdot\eta\geq 0$ on $\partial \Omega$ and that in particular
\begin{equation*}
x\cdot\eta>0\quad \text{on}\,\, \partial\Omega_1\subseteq\partial \Omega,
\end{equation*}
where $\partial \Omega_1$ is some subset of the boundary of positive measure. Finally  we call a domain  strictly star-shaped with respect to the origin,  if  $x\cdot\eta>0$ on the whole boundary  $\partial \Omega$.

Thanks to Theorem \ref{thm:pohozaev} we can state a general Pohozaev's non-existence type theorem in the spirit of \cite[Theorem 1]{pucser1}, for the anisotropic Finsler operator. We have the following
\begin{thm}[Pohozaev-Pucci-Serrin  non-existence  type theorem]\label{thm:PPS}
Let us assume that $(h_g)$, $(h_B)$ and $(h_H)$ hold and  that $\Omega$ is smooth, bounded and star-shaped with respect to the origin. Suppose that
\begin{eqnarray}\label{eq:shf1}
&&G(x,0)-B(H(\xi)) +B'(H(\xi))H(\xi)\geq0
\\\nonumber
&& \text{for all}\,\, x\in \partial \Omega\,\, \text{and}\,\, \xi\in\mathbb R^N
\end{eqnarray}
so that
\begin{equation}\label{eq:shf2}
\int_{\Omega}NG(x,u)  + (x\cdot \nabla_x G(x,u))-NB(H(\nabla u)) +B'(H(\nabla u))H(\nabla u)\, dx\geq 0.
\end{equation}
Finally assume that  \eqref{eq:shf2} implies either $u=0$ or $\nabla u=0$. Then the equation \eqref{eq:p} has no nontrivial weak solution $u\in C^1(\overline\Omega)$ with zero Dirichlet boundary condition.
\end{thm}
\begin{proof}
Applying Theorem \ref{eq:pohozaev}, any weak $C_1$ solution to \eqref{eq:p} with zero Dirichlet boundary condition satisfies the following
\begin{eqnarray}\label{eq:shf3}
&& \int_{\Omega}NG(x,u)+ (x\cdot \nabla_x G(x,u))-NB(H(\nabla u)) +B'(H(\nabla u))H(\nabla u) \, dx
\\\nonumber
&&=\int_{\partial \Omega}G(x,0)(x\cdot \eta) -B(H(\nabla u))(x\cdot\eta) +B'(H(\nabla u))(x\cdot \nabla u) (\nabla H(\nabla u)\cdot \eta)\, ds.
\end{eqnarray}
Using the fact that $u=0$ on $\partial \Omega$ and therefore $\nabla u(x)=u_{\eta}(x)\eta$ there, we get
\begin{eqnarray}
&&\int_{\partial \Omega}B'(H(\nabla u))(x\cdot \nabla u) (\nabla H(\nabla u)\cdot \eta)\, ds\\\nonumber
&&=\int_{\partial \Omega}B'(H(\nabla u))(x\cdot \eta)u_{\eta} (\nabla H(u_{\eta}\eta)\cdot \eta)\, ds\\\nonumber
&&=-\int_{\partial \Omega}B'(H(\nabla u))(x\cdot \eta)u_{\eta} (\nabla H(\eta)\cdot \eta)\, ds\\\nonumber
&&=\int_{\partial \Omega}B'(H(\nabla u))H(\nabla u)(x\cdot \eta)\, ds,
\end{eqnarray}
where (see hypotheses $(h_H)$) we used Euler's theorem and that for  $s\in \mathbb R$ we have $\nabla H(s\xi)=\text{sign}(s) \nabla H(\xi)$ (i.e. $\nabla H$ is  absolutely $0$-homogeneous function).  Hence the right hand side of \eqref{eq:shf3} reads as
\[\int_{\partial \Omega}\big(G(x,0) -B(H(\nabla u))+B'(H(\nabla u))H(\nabla u)\big)(x\cdot \eta) \, ds.
\]
Then, since the domain is star-shaped, by \eqref{eq:shf1} the right hand side of \eqref{eq:shf3} is indeed non-negative and therefore \eqref{eq:shf2} holds true. By assumption,  either $u=0$ or $\nabla u=0$ in $\Omega$. Therefore in any case, by Stampacchia's theorem (see  for instance \cite[Lemma 7.7]{GT}) we deduce that $\nabla u=0$ a.e. in $\Omega$. Therefore, since $u=0$ on $\partial \Omega$ we have that  $u\equiv 0$ in $\Omega$.
\end{proof}
The hypothesis of Theorem \ref{thm:PPS}, namely  that  \eqref{eq:shf2} implies either $u=0$ or $\nabla u=0$ may not always be applicable. This is the case of the critical problem
\eqref{eq:critprob} here below: indeed,  for such a problem,  the left hand side of \eqref{eq:shf2}  is identically zero. Neverthless  as a  further and  important consequence of
  Theorem \ref{thm:pohozaev},  we get  the non existence of nontriavial solutions in bounded smooth star-shaped domain $\Omega$,  of the critical  anisotropic p-Laplacian problem, i.e.
\begin{equation}\label{eq:critprob}
\begin{cases}
-\Delta^H_p u=|u|^{p^*-2}u& \text {in}\,\, \Omega\\
u=0 & \text{on}\,\, \partial \Omega,
\end{cases}
\end{equation}
where $-\Delta^H_p$ is the anisotropic operator defined for suitable smooth functions $u$ by
\begin{equation}\label{eq:anisotropic}-\Delta_p^Hu:=-\operatorname{div}(H^{p-1}(\nabla u)\nabla H(\nabla u)),
\end{equation}
and $p^*=Np/(N-p)$, $1<p<N$.
We have
\begin{cor}\label{cor:pohozaev}Let us assume that $(h_g)$ and $(h_H)$ hold.  Let $\Omega$ a bounded smooth domain strictly star-shaped with respect to the origin.  Any $C^1(\overline\Omega)$ solution to
\begin{equation}\label{eq:probl}
\begin{cases}
-\Delta^H_p u=g(u)& \text {in}\,\, \Omega\\
u=0 & \text{on}\,\, \partial \Omega,
\end{cases}
\end{equation}
such that $\nabla u\not \equiv 0$ on $\partial \Omega$, satisfies
\begin{equation}\label{eq:pohozaev1}
N\int_{\Omega}G(u)\, dx-\frac{N-p}{p}\int_{\Omega}g(u)u \, dx>0.
\end{equation}

\

In particular, in the special case of  $g(u)=|u|^{m-1}u$, $m>0$ and  $1<p<N, $ then \eqref{eq:pohozaev1} still holds and we obtain
\begin{equation}\label{eq:pohozaev2}
\left(\frac{N}{m+1}-\frac{N-p}{p}\right)\int_{\Omega}|u|^{m+1} \, dx>0.
\end{equation}
Therefore $u\not\equiv 0$ in $\Omega$ implies
\[m<\frac{N(p-1)+p}{N-p}:=p^*-1.\]
\end{cor}
\begin{proof}
Using \eqref{eq:G} in the special case $g(x,t)=g(t)$ and  $B(t)=t^p/p$, $p>1$,
since $u=0$ on $\partial \Omega$, from  \eqref{eq:pohozaev} we obtain
\begin{eqnarray}\label{eq:pohozaev_pro}
&&\qquad N\int_{\Omega}G(u)\, dx-\frac{N}{p}\int_{\Omega}H^p(\nabla u) \, dx +\int_{\Omega}H^p(\nabla u) \, dx
\\\nonumber
&&=-\frac 1p\int_{\partial \Omega} H^p(\nabla u)(x\cdot\eta)\, ds + \int_{\partial \Omega}H^{p-1}(\nabla u)(x\cdot \nabla u) (\nabla H(\nabla u)\cdot \eta)\, ds.
\end{eqnarray}
Moreover, on the boundary $\partial \Omega$, ($u=0$ there)  we have that $\nabla u(x)=u_{\eta}(x)\eta$. Arguing as in the proof of Theorem \ref{thm:PPS},  the last term on the right hand side of \eqref{eq:pohozaev_pro} reads as
\begin{eqnarray}\label{eq:civuole}
&&\int_{\partial \Omega}H^{p-1}(\nabla u)(x\cdot \nabla u) (\nabla H(\nabla u)\cdot \eta)\, ds\\\nonumber
&&=\int_{\partial \Omega}H^{p-1}(u_{\eta}(x)\eta)(x\cdot u_{\eta}(x)\eta) (\nabla H(u_{\eta}(x)\eta)\cdot \eta)\, ds\\\nonumber
&&=-\int_{\partial \Omega}H^{p-1}(u_{\eta}(x)\eta))u_{\eta}(x)(x\cdot \eta) (\nabla H(\eta)\cdot \eta)\, ds
\\\nonumber
&&
=-\int_{\partial \Omega}H^{p-1}(u_{\eta}(x)\eta)u_{\eta}(x)(x\cdot \eta) H(\eta)\, ds\\\nonumber
&&=\int_{\partial \Omega}H^{p}(u_{\eta}(x)\eta))(x\cdot \eta)\, ds
\\\nonumber
&&=\int_{\partial \Omega}H^p(\nabla u)(x\cdot \eta)\, ds.
\end{eqnarray}
Therefore since $\Omega$ is strictly star-shaped with respect to the origin, i.e. $(x,\eta)>0$ on the boundary $\partial \Omega$, exploiting  \eqref{eq:civuole}, we deduce that the  right hand side of \eqref{eq:pohozaev_pro} is indeed positive, namely
\begin{eqnarray}\label{eq:rhs}
&&-\frac 1p\int_{\partial \Omega} H^p(\nabla u)(x\cdot\eta)\, ds + \int_{\partial \Omega}H^{p-1}(\nabla u)(x\cdot \nabla u) (\nabla H(\nabla u)\cdot \eta)\, ds\\\nonumber
&&=-\frac 1p\int_{\partial \Omega} H^p(\nabla u)(x\cdot\eta)\, ds  +\int_{\partial \Omega}H^p(\nabla u)(x\cdot \eta)\, ds\\\nonumber
 &&=\frac{p-1}{p}\int_{\partial \Omega}H^p(\nabla u)(x\cdot \eta)\, ds>0,
\end{eqnarray}
where in the last line we used assumptions $(h_H)$ and  that $\nabla u\not \equiv 0$ on $\partial \Omega$.

Using the solution $u$ as test function in the problem \eqref{eq:probl}, the right hand side of \eqref{eq:pohozaev_pro} becomes
\begin{eqnarray}\label{eq:rhspohoz}
&&N\int_{\Omega}G(u)\, dx-\frac{N}{p}\int_{\Omega}H^p(\nabla u) \, dx +\int_{\Omega}H^p(\nabla u) \, dx\\\nonumber
&&=N\int_{\Omega}G(u)\, dx-\frac{N-p}{p}\int_{\Omega}g(u)u \, dx.
\end{eqnarray}
Putting together  \eqref{eq:rhs} and \eqref{eq:rhspohoz} in \eqref{eq:pohozaev_pro} we get  \eqref{eq:pohozaev1}.

Readily  from \eqref{eq:pohozaev2},  we deduce that $u\not\equiv 0$ in $\Omega$ implies $0<m<p^*-1$.
\end{proof}
Finally we give  the
\begin{proof}[{Proof of Theorem~\ref{thm:R^N}}]
Using polar coordinates and \eqref{eq:radii}, we infer there exists  a sequence of radii $\{r_n\}$   such that  for $n\rightarrow+\infty$ we have
\begin{equation}\label{eq:radii2}
\lim_{n\rightarrow +\infty}r_n\int_{\partial B(0,r_n)}(B(H(\nabla u))+|G(x,u)|)\, dS=0.
\end{equation}
By Theorem \ref{thm:pohozaev} with $\Omega=B(0,r_n)$ we obtain
\begin{eqnarray}\label{eq:sianiua}\\\nonumber
&& \int_{B(0,r_n)}NG(x,u)- (x\cdot \nabla_x G(x,u))\, dx-\int_{B(0,r_n)}NB(H(\nabla u)) -B'(H(\nabla u))H(\nabla u) \, dx
\\\nonumber
&&=\int_{\partial B(0,r_n)}G(x,u)(x\cdot \eta)\, ds-\int_{\partial B(0,r_n)} B(H(\nabla u))(x\cdot\eta)\, ds \\\nonumber
&&+ \int_{\partial B(0,r_n)}B'(H(\nabla u))(x\cdot \nabla u) (\nabla H(\nabla u)\cdot \eta)\, ds.
\end{eqnarray}
Exploiting hypotheses in $(h_B)$ we deduce that $B'(t)t\leq CB(t)$ for all $t\geq 0$ and with  $C=C(p,\gamma,\Gamma)$  a positive constant.
Then, recalling also \eqref{eq:normequivbag} we deduce
\begin{eqnarray}\nonumber
&&\left|\int_{\partial B(0,r_n)}B'(H(\nabla u))(x\cdot \nabla u) (\nabla H(\nabla u)\cdot \eta)\, ds\right|\\\nonumber
&&\leq C r_n\int_{\partial B(0,r_n)}B'(H(\nabla u))H(\nabla u)\, ds\leq C r_n\int_{\partial B(0,r_n)}B(H(\nabla u))\, ds=o(1),
\end{eqnarray}
for $n\rightarrow +\infty$, because of \eqref{eq:radii2}.

The other terms in the right hand side of \eqref{eq:sianiua} clearly go to zero as well. Therefore, passing to the limit in \eqref{eq:sianiua}, by the dominate convergence theorem we end the proof.
\end{proof}

\vspace{1cm}

\begin{center}{\bf Acknowledgements}\end{center}  L. Montoro and B. Sciunzi are partially supported by PRIN project 2017JPCAPN (Italy): Qualitative and quantitative aspects of nonlinear PDEs, and L. Montoro by  Agencia Estatal de Investigación (Spain), project PDI2019-110712GB-100.

\

\begin{center}
{\sc Data availability statement}

\

All data generated or analyzed during this study are included in this published article.
\end{center}

\end{document}